\documentclass[twoside,10pt]{article}

\usepackage{a4}                       
\usepackage[english]{babel}             
\usepackage[T1]{fontenc}              
\usepackage[latin1]{inputenc}         
\usepackage{amsmath,amsfonts,amssymb} 
\usepackage{graphics,graphicx}        
\usepackage{subfigure}                
\usepackage{epsfig}
\usepackage{ae,aecompl}
\usepackage{color}

\def\R{\mathbb{R}}

\def\Z{\mathbb{Z}}
\def\N{\mathbb{N}}
\def\C{\mathbb{C}}

\def\k{\textit {\textbf k}}

\newcommand\cR{\mathcal R}

\newcommand\Ec{{E_{\rm c}}}
\newcommand\cN{\mathcal N}
\newcommand\cI{\mathcal I}
\newcommand\cG{\mathcal G}

\newcommand\dps{\displaystyle }

\newtheorem{theorem}{Theorem}[section]
\newtheorem{lemma}[theorem]{Lemma}
\newtheorem{e-proposition}[theorem]{Proposition}

\newtheorem{e-definition}[theorem]{Definition\rm}

\newtheorem{remark}{Remark} 

\def\qed{\relax
     \ifmmode
       ~\hfill\Box
     \else
        \unskip\nobreak ~\hfill$\Box$
      \fi \par}

\newtheorem{Proof}{Proof}

\newenvironment{proof}{\begin{Proof}\rm}{\qed\end{Proof}}

\begin{document}

\title{Numerical analysis of the planewave discretization of
  orbital-free and Kohn-Sham models \\
Part I: The Thomas-Fermi-von Weizs\"acker model}
\author{Eric Canc\`es\footnote{Universit\'e Paris-Est, CERMICS,
    Project-team Micmac, INRIA-Ecole des Ponts,  6 \& 8 avenue Blaise
    Pascal, 77455 Marne-la-Vall\'ee Cedex 2, France.}, Rachida
  Chakir\footnote{UPMC Univ Paris 06, UMR 7598 LJLL, Paris, F-75005 France ;
CNRS, UMR 7598 LJLL, Paris, F-75005 France} 
$\,$ and Yvon Maday$^\dag$\footnote{Division of Applied Mathematics, Brown
  University, Providence, RI, USA} } 
%
%
\maketitle

\begin{abstract}
We provide {\it a priori} error estimates for the spectral and
pseudospectral Fourier (also called planewave) discretizations of the
periodic Thomas-Fermi-von Weizs\"acker (TFW) model and of the Kohn-Sham
model, within the local density approximation (LDA). These models
allow to compute approximations of the ground state energy and density
of molecular systems in the condensed phase. The TFW model is stricly
convex with respect to the electronic density, and allows for a
comprehensive analysis (Part I). This is not the case for the Kohn-Sham LDA
model, for which the uniqueness of the ground state electronic density
is not guaranteed. Under a coercivity assumption on the second order
optimality condition, we prove in Part~II that for large enough energy
cut-offs, 
the discretized Kohn-Sham LDA problem has a minimizer in the
vicinity of any Kohn-Sham ground state, and that this minimizer is unique up to
unitary transform. We then derive optimal  {\it a priori} error estimates
for both the spectral and the pseudospectral discretization methods.
\end{abstract}

\selectlanguage{english}
\section{Introduction}
\label{sec:introduction}

Density Functional Theory (DFT) is a powerful method
for computing ground state electronic energies and
densities in quantum chemistry, materials science, molecular biology and
nanosciences. The models originating from DFT can be classified into two
categories: the orbital-free models and the Kohn-Sham models. 
The Thomas-Fermi-von Weizs\"acker (TFW) model falls into the first
category. It is not very much used in practice, but is
interesting from a mathematical viewpoint. It indeed serves as a toy
model for the analysis of the more complex electronic structure models
routinely used by Physicists and Chemists. At the other extremity of the
spectrum, the Kohn-Sham models are among the most widely used models in
Physics and Chemistry, but are much more difficult to deal with. We focus
here on the numerical analysis of the TFW model on the one hand, and of
the Kohn-Sham model, within the local density approximation (LDA), on
the other hand. More precisely, we are interested in the
pseudospectral Fourier, more commonly called planewave, discretization
of the periodic version of these two models. In this context, the
simulation domain, sometimes refered to as the supercell, is the unit cell
of some periodic lattice of $\R^3$. In the TFW model, periodic boundary
conditions 
(PBC) are imposed to the density; in the Kohn-Sham framework, they are
imposed to the Kohn-Sham orbitals (Born-von Karman PBC). Imposing
PBC at the boundary of the simulation cell is a standard method to
compute condensed phase properties with a limited number of atoms in the
simulation cell, hence at a moderate computational cost.
 
This article is organized as follows. In Section~\ref{sec:Fourier}, we
briefly introduce the functional setting used in the formulation and
the analysis of the planewave discretization of orbital-free and
Kohn-Sham models. In Section~\ref{sec:TFW}, we provide {\it a priori}
error estimates for the planewave discretization of the TFW model. Our
estimates refine and complement some of the results given in
\cite{Zhou_TFW}. In Part II, we deal with the Kohn-Sham LDA
model.

\section{Basic Fourier analysis for planewave discretization methods}
\label{sec:Fourier}

Throughout this article, we denote by $\Gamma$ the simulation cell, by
$\cR$ the periodic lattice, and by
$\cR^\ast$ the dual lattice. For
simplicity, we assume that
$\Gamma=[0,L)^3$  ($L > 0$), but our arguments can be straightforwardly
extended to rectangular simulation cells ($\Gamma=[0,L_x) \times
[0,L_y) \times [0,L_z)$). For  $\Gamma=[0,L)^3$, $\cR$ is the cubic
lattice $L\Z^3$, and
$\cR^\ast = \frac{2\pi}L \Z^3$. For $k \in
\cR^\ast$, we denote by $e_k(x)=|\Gamma|^{-1/2} \, e^{ik\cdot x}$ the
planewave with wavevector $k$. The family $(e_k)_{k \in \cR^\ast}$
forms an orthonormal basis of 
$$
L^2_\#(\Gamma,\C):=\left\{ u \in L^2_{\rm loc}(\R^3,\C) \; | \; u \mbox{
    $\cR$-periodic} \right\},
$$
and for all $u \in L^2_\#(\Gamma,\C)$,
$$
u(x) = \sum_{k \in \cR^\ast} \widehat u_k \, e_k(x) \qquad \mbox{with} \qquad
\widehat u_k=(e_k,u)_{L^2_\#} = |\Gamma|^{-1/2} \int_\Gamma u(x)
e^{-ik\cdot x} \, dx.
$$
In our analysis, we will only consider real valued functions. We
therefore introduce the Sobolev spaces of real valued functions 
$$
H^s_\#(\Gamma) :=\left\{ u(x) =  \sum_{k \in \cR^\ast} \widehat u_k \,
  e_k(x) \; | \; \sum_{k \in \cR^\ast} (1+|k|^2)^s |\widehat
  u_k|^2 < \infty \mbox{ and } \forall k,  \; c_{-k}=c_k^\ast  \right\},
$$
$s \in \R$, endowed with the inner products
$$
(u,v)_{H^s_\#} =  \sum_{k \in \cR^\ast} (1+|k|^2)^s \,
\widehat u_k^\ast \, \widehat v_k.
$$ 
For $N_c \in \N$, we denote by
\begin{equation} 
\label{eq_1}
V_{N_c} = \left\{ \sum_{k \in \cR^\ast \, | \, |k| \le \frac{2\pi}L N_c} c_k
  e_k \; | \; \forall k, \; c_{-k}=c_k^\ast \right\}
\end{equation} 
(the constraints $c_{-k}=c_k^\ast$ imply that the functions of $V_{N_c}$
are real valued). For all $s \in \R$, and each $v
\in H^s_\#(\Gamma)$, the best approximation of $v$ in $V_{N_c}$ for {\it any}
$H^r_\#$-norm, $r \le s$, is
$$
\Pi_{N_c} v = \sum_{\k\in \cR^\ast \, | \, |\k| \le \frac{2\pi}L N_c}
\widehat v_k e_k. 
$$
The more regular $v$ (the regularity being measured in terms of the
Sobolev norms $H^r$), the faster the convergence of this truncated
series to $v$: for all real numbers $r$ and $s$ with $r \le s$, we
have for each $v\in H^s_\#(\Gamma)$,
\begin{eqnarray}
\|v - \Pi_{N_c}v\|_{H^r_\#} = \min_{v_{N_c}
  \in V_{N_c}} \|v-v_{N_c}\|_{H^r_\#} & \le & 
\left( \frac{L}{2\pi} \right)^{s-r} \, N_c^{-(s-r)} \|v-
\Pi_{N_c}v\|_{H^s_\#} 
\nonumber \\
&\le & 
\left( \frac{L}{2\pi} \right)^{s-r} \, N_c^{-(s-r)} \|v\|_{H^s_\#}.
\label{eq:app-Fourier}
\end{eqnarray}
For $N_g \in \N \setminus \left\{0\right\}$, we denote by 
$\widehat{\phi}^{{\rm FFT},N_g}$ the discrete Fourier transform 
on the carterisan grid $\cG_{N_g}:=\frac{L}{N_g} \, \Z^3$ of the
function $\phi \in C^0_\#(\Gamma)$. Recall that
if $\phi = \sum_{k \in \cR^\ast} \widehat \phi_g \, e_g \in C^0_\#(\Gamma)$,
the discrete Fourier transform of $\phi$ is the $N_g\cR^\ast$-periodic
sequence $\widehat{\phi}^{{\rm FFT},N_g}=(\widehat{\phi}^{{\rm FFT},N_g}_{k})_{k
  \in \cR^\ast}$ where  
$$
\widehat{\phi}^{{\rm FFT},N_g}_{k}
= \frac{1}{N_g^3} \sum_{x \in \cG_{N_g} \cap \Gamma} \phi(x) e^{-ik
  \cdot x} = |\Gamma|^{-1/2} \sum_{K \in \cR^\ast} \widehat \phi_{k+N_gK}.
$$
We now introduce the subspaces 
$$
W_{N_g}^{\rm 1D}  \; = \; \left| \begin{array}{lll}
\dps \mbox{Span} \left\{ e^{ily} \; | \; l \in \frac{2\pi}L \Z, \; |l| \le 
\frac{2\pi}L \left( \frac{N_g-1}2 \right) \right\}
 & \quad  (N_g \mbox{ odd}), \\
 \dps \mbox{Span} \left\{ e^{ily} \; | \; l \in \frac{2\pi}L \Z, \; |l| \le 
\frac{2\pi}L \left( \frac{N_g}2 \right) \right\} \oplus \C
(e^{i \pi N_gy/L}+e^{-i \pi N_gy/L}) 
& \quad  (N_g \mbox{ even}),
\end{array} \right.
$$
($W_{N_g}^{\rm 1D} \in C^\infty_\#([0,L))$ and $\mbox{dim}(W_{N_g}^{\rm
  1D}) = N_g$), and 
$W_{N_g}^{\rm 3D} = W_{N_g}^{\rm 1D} \otimes W_{N_g}^{\rm 1D} \otimes
W_{N_g}^{\rm 1D}$.
Note that $W_{N_g}^{\rm 3D}$ is a subspace of $H^s_\#(\Gamma)$ of dimension
$N_g^3$, for all $s \in \R$, and that if $N_g$ is odd,
$$
W_{N_g}^{\rm 3D} =  \mbox{Span} \left\{ e_k \; | \; k \in \cR^\ast =
  \frac{2\pi}L \Z^3, \;
  |k|_\infty \le  \frac{2\pi}L \left( \frac{N_g-1}2 \right) \right\}
\qquad  \qquad  (N_g \mbox{ odd}).
$$
It is then possible to define
the interpolation projector $\cI_{N_g}$ from $C^0_\#(\Gamma)$ onto
$W_{N_g}^{\rm 3D}$ by $[\cI_{N_g}(\phi)](x) = \phi(x)$ for all 
$x \in \cG_{N_g}$. It holds
\begin{equation} \label{eq:integration_formula_INg}
\forall \phi \in C^0_\#(\Gamma), \quad 
\int_\Gamma \cI_{N_g}(\phi) =  \sum_{x \in \cG_{N_g} \cap \Gamma} \left(
  \frac{L}{N_g} \right)^3 \phi(x).
\end{equation}
The coefficients of the expansion of
$\cI_{N_g}(\phi)$ in the canonical 
basis of $W_{N_g}^{\rm 3D}$ is given by the discrete Fourier transform of
$\phi$. In particular, when $N_g$ is odd, we have the simple relation
$$
\cI_{N_g}(\phi)  =  
|\Gamma|^{1/2} \dps \sum_{k \in \cR^\ast \, | \, 
|k|_\infty \le \frac{2\pi}L \left( \frac{N_g-1}2 \right)} \widehat{\phi}^{{\rm
      FFT},N_g}_{k} \, e_k  \qquad \qquad  (N_g \mbox{ odd}).
$$
It is easy to check that if $\phi$ is real-valued, then so is
$\cI_{N_g}(\phi)$.

We will assume in the sequel that $N_g \ge 4N_c+1$. We will then have for
all $v_{4N_c} \in V_{4N_c}$, 
\begin{equation} \label{eq:exact_integration}
\int_\Gamma v_{4N_c} = \sum_{x \in \cG_{N_g} \cap \Gamma}
\left( \frac{L}{N_g} \right)^3 v_{4N_c}(x) = \int_\Gamma \cI_{N_g} (v_{4N_c}).
\end{equation}

The following lemma gathers some technical results which will be useful
for the numerical analysis of the planewave discretization of
orbital-free and Kohn-Sham models.

\begin{lemma} \label{lem:technical}
Let $N_c \in \N^\ast$ and $N_g \in \N^\ast$ such that $N_g \ge 4N_c+1$.
  \begin{enumerate}
  \item Let $V$ be a real-valued function of $C^0_\#(\Gamma)$ and $v_{N_c}$
    and $w_{N_c}$ be two functions of $V_{N_c}$. Then 
\begin{eqnarray}
\int_\Gamma \cI_{N_g}(V v_{N_c}w_{N_c}) &=& 
\int_\Gamma \cI_{N_g}(V) v_{N_c} w_{N_c}  \label{eq:ineg_INg_2} \\
\left| \int_\Gamma \cI_{N_g}(V |v_{N_c}|^2) \right| &\le & \|V\|_{L^\infty} 
\|v_{N_c}\|_{L^2_\#}^2  \label{eq:bound_INg1}
\end{eqnarray}

  \item Let $s > 3/2$, $0 \le r \le s$, and $V$ a function of
    $H^s_\#(\Gamma)$. Then,
\begin{eqnarray}
\left\| (1-\cI_{N_g})(V) \right\|_{H^r_\#} & \le &
C_{r,s} N_g^{r-s} \|V\|_{H^s_\#}  \label{eq:ineg_INg_0} \\
\left\| \Pi_{2N_c}(\cI_{N_g}(V)) \right\|_{L^2_\#} & \le &
\left( \int_\Gamma \cI_{N_g}(|V|^2) \right)^{1/2}  \label{eq:ineg_INg_4} \\
\left\| \Pi_{2N_c}(\cI_{N_g}(V)) \right\|_{H^s_\#} & \le &
(1+C_{s,s}) \|V\|_{H^s_\#}  \label{eq:ineg_INg_5}
\end{eqnarray}
for constants $C_{r,s}$ independent of $V$. Besides if there exists $m
> 3$ and $C \in \R_+$ such that $|\widehat V_k| \le C |k|^{-m}$, then
there exists a constant $C_V$ independent of $N_c$ and $N_g$ such that
\begin{eqnarray}
\left\| \Pi_{2N_c}(1-\cI_{N_g})(V) \right\|_{H^r} & \le & 
C_V N_c^{r+3/2} N_g^{-m}  \label{eq:ineg_INg_6} 
\end{eqnarray}
 \item Let $\phi$ is be a Borel function from
$\R_+$ to $\R$ such that there exists $C_\phi \in \R_+$ for which
$|\phi(t)| \le C_\phi (1+t^2)$ for all $t \in \R_+$. Then, for all $v_{N_c}
\in V_{N_c}$,
\begin{eqnarray}
\left| \int_\Gamma \cI_{N_g}(\phi(|v_{N_c}|^2)) \right| &\le & C_\phi \left(
  |\Gamma|+\|v_{N_c}\|_{L^4_\#}^4 \right). \label{eq:bound_INg2} 
\end{eqnarray}
  \end{enumerate}
\end{lemma}

\begin{proof} 
For $z_{2N_c} \in V_{2N_c}$, it holds
\begin{eqnarray}
\int_\Gamma \cI_{N_g}(V z_{2N_c}) &=& \sum_{x \in \cG_{N_g} \cap \Gamma}
V(x) z_{2N_c}(x)  \nonumber \\
& = & \sum_{x \in \cG_{N_g} \cap \Gamma}
(\cI_{N_g}(V))(x) z_{2N_c}(x)  \nonumber \\
& = & \int_\Gamma \cI_{N_g}(V) \,
z_{2N_c} \label{eq:equality_INg}
\end{eqnarray}
since $\cI_{N_g}(V) z_{2N_c} \in V_{4N_c}$. The function
$v_{N_c}w_{N_c}$ being in $V_{2N_c}$, (\ref{eq:ineg_INg_2}) 
is proved. Moreover, as $|v_{N_c}|^2 \in V_{4N_c}$, it
follows from (\ref{eq:exact_integration}) that 
\begin{eqnarray*}
\left| \int_\Gamma \cI_{N_g}(V |v_{N_c}|^2 ) \right| & = &
\left|  \sum_{x \in \cG_{N_g} \cap \Gamma} \left( \frac{L}{N_g}
  \right)^3  V(x)  |v_{N_c}(x)|^2 \right| \\
& \le &  \|V\|_{L^\infty} \left|  \sum_{x \in \cG_{N_g}
    \cap \Gamma} \left( \frac{L}{N_g} 
  \right)^3    |v_{N_c}(x)|^2 \right| \\
& = &  \|V\|_{L^\infty} \int_\Gamma |v_{N_c}|^2.
\end{eqnarray*}
Hence (\ref{eq:bound_INg1}). 
The estimate (\ref{eq:ineg_INg_0}) is
proved in \cite{CHQZ}. To prove (\ref{eq:ineg_INg_4}), we
notice that
\begin{eqnarray*}
\|\Pi_{2N_c}(I_{N_g}(V))\|_{L^2_\#}^2 & \le  & \|I_{N_g}(V)\|_{L^2_\#}^2
\\
& = & \int_\Gamma (I_{N_g}(V))^\ast  (I_{N_g}(V)) \\
& = &  \sum_{x \in \cG_{N_g} \cap \Gamma} (I_{N_g}(V))(x)^\ast
(I_{N_g}(V))(x) \\
& = &  \sum_{x \in \cG_{N_g} \cap \Gamma} |V(x)|^2 \\
& = &  \int_\Gamma I_{N_g}(|V|^2).
\end{eqnarray*}
The bound (\ref{eq:ineg_INg_5}) is a straightforward
consequence of (\ref{eq:ineg_INg_0}):
\begin{eqnarray*}
\|\Pi_{2N_c}(I_{N_g}(V))\|_{H^s_\#} & \le & \|I_{N_g}(V)\|_{H^s_\#} \le 
\|V\|_{H^s_\#} + \|(1-I_{N_g})(V)\|_{H^s_\#}  \le 
(1+C_{s,s}) \|V\|_{H^s_\#}.
\end{eqnarray*}
Now, we notice that
\begin{eqnarray}
\Pi_{2N_c} (\cI_{N_g}(V)) &=& |\Gamma|^{1/2} 
\sum_{k \in \cR^\ast \, | \, |k| \le \frac{4\pi}{L}N_c} \widehat
V_k^{{\rm FFT},N_g} e_k 
\nonumber \\
& = &  \sum_{k \in \cR^\ast \, | \, |k| \le \frac{4\pi}{L}N_c} \left(
  \sum_{K \in \cR^\ast}  
\widehat V_{k+N_g K} \right)  e_k \label{eq:Pi2NINg}.
\end{eqnarray}
From (\ref{eq:Pi2NINg}), we obtain
\begin{eqnarray*}
\left\| \Pi_{2N_c}(1-\cI_{N_g})(V) \right\|_{H^s}^2 & = &
 \sum_{k \in \cR^\ast \, | \, |k| \le \frac{4\pi}{L}N_c} (1+|k|^2)^s \left| 
\sum_{K \in \cR^\ast \setminus \left\{0\right\}} 
\widehat V_{k+N_g K} \right|^2 \\ & = & \left( \sum_{k \in \cR^\ast \, |
  \, |k| \le \frac{4\pi}{L}N_c} (1+ |k|^2)^s \right) 
 \max_{k \in \cR^\ast \, | \, |k| \le \frac{4\pi}{L}N_c}  \left| 
\sum_{K \in \cR^\ast \setminus \left\{0\right\}} 
\widehat V_{k+N_g K} \right|^2 .
\end{eqnarray*}
On the one hand,
$$
\sum_{k \in \cR^\ast \, | \, 
|k| \le \frac{4\pi}{L}N_c} (1+ |k|^2)^s \mathop{\sim}_{N_c \to
  \infty} \frac{32\pi}{2s+3} \left( \frac{4\pi}L \right)^{2s} \, N_c^{2s+3},  
$$
and on the other hand, we have for each $k \in \cR^\ast$ such that $|k|
\le \frac{4\pi}{L}N_c$, 
\begin{eqnarray*}
 \left| \sum_{K \in \cR^\ast \setminus \left\{0\right\}} 
\widehat V_{k+N_g K} \right| & \le & C 
 \sum_{K \in \cR^\ast \setminus \left\{0\right\}} \frac{1}{|k+N_gK|^m}
 \\
& \le & C_0 \left( \frac{L}{2\pi} \right)^{m} N_g^{-m} 
\end{eqnarray*}
where
$$
C_0 = \max_{y \in \R^3 \, | \, |y| \le 1/2} \sum_{K \in \Z^3 \setminus
  \left\{0\right\}}  \frac{1}{|y-K|^m}.
$$
The estimate (\ref{eq:ineg_INg_6}) then easily follows. Let us finally prove
(\ref{eq:bound_INg2}). Using  (\ref{eq:integration_formula_INg}) and
(\ref{eq:exact_integration}), we have
\begin{eqnarray*}
\left| \int_\Gamma \cI_{N_g}(\phi(|v_N|^2)) \right| &=&
\left| \sum_{x \in \frac{L}{N_g} \Z^3 \cap \Gamma} \left( \frac{L}{N_g}
  \right)^3  \phi(|v_N(x)|^2) \right| \\
& \le & C_\phi
\left| \sum_{x \in \frac{L}{N_g} \Z^3 \cap \Gamma} \left( \frac{L}{N_g}
  \right)^3  (1+|v_N(x)|^4) \right| \\
& = &  C_\phi \int_\Gamma (1+|v_N|^4)  = 
C_\phi \left( |\Gamma|+\|v_N\|_{L^4_\#}^4 \right).
\end{eqnarray*}
This completes the proof of Lemma~\ref{lem:technical}.
\end{proof}

\section{Thomas-Fermi-von-Weizs\"acker model}
\label{sec:TFW}

In the TFW model, as well as in any orbital-free model, the ground state
electronic density of the system is obtained by minimizing an explicit
functional of the density. Denoting by $\cN$ the number of electrons in
the simulation cell and by
$$
{\mathfrak R}_\cN = \left\{ \rho \ge 0 \; | \; \sqrt\rho \in
  H^1_\#(\Gamma), \; \int_\Gamma \rho = \cN \right\}
$$
the set of admissible densities, the TFW problem reads
\begin{equation} 
I^{\rm TFW}  =  \inf \left\{ {\mathcal E}^{\rm TFW}(\rho), \; \rho \in
  {\mathfrak R}_\cN \right\} \label{eq:minTFWrho}, 
\end{equation}
where 
$$\dps
{\mathcal E}^{\rm TFW}(\rho)  =  \frac{C_{\rm W}}2 \int_\Gamma |\nabla
\sqrt \rho|^2 +  
C_{\rm TF} \int_\Gamma \rho^{5/3} + \int_\Gamma \rho V^{\rm ion} + \frac
12 D_\Gamma(\rho,\rho).
$$

\noindent 
$C_{\rm W}$ is a positive real number ($C_{\rm W}=1$, $1/5$ or $1/9$
depending on the context \cite{DreizlerGross}), and $C_{\rm TF}$ is 
the Thomas-Fermi constant: $C_{\rm TF}=\frac{10}3(3\pi^2)^{2/3}$.
The last term of the TFW energy models the periodic Coulomb energy: for
$\rho$ and $\rho'$ in $H^{-1}_\#(\Gamma)$,
$$
D_\Gamma(\rho,\rho'):= 4 \pi \sum_{k \in \cR^\ast \setminus
  \left\{0\right\}} |k|^{-2} \widehat \rho_k^\ast \, 
\widehat \rho'_k.
$$
We finally make the assumption that $V^{\rm ion}$ is a periodic 
potential such that 
\begin{equation} \label{eq:hypothesis}
\exists m > 3, \; C \ge 0 \mbox{ s.t. }  \forall k \in \cR^\ast, \;
|\widehat V^{\rm ion}_k| \le C |k|^{-m}.
\end{equation}
Note that this implies that $V^{\rm ion}$ is in
$H^{m-3/2-\epsilon}(\Gamma)$ for all $\epsilon > 0$. 
It is convenient to reformulate the TFW model in terms of
$v=\sqrt{\rho}$. It can be seen that
\begin{equation} 
I^{\rm TFW}  =  \inf \left\{ E^{\rm TFW}(v), \; v \in H^1_\#(\Gamma),
  \; \int_\Gamma |v|^2 = \cN \right\} \label{eq:minTFWu} 
\end{equation}
where 
$$
E^{\rm TFW}(v)  =  \frac{C_{\rm W}}2 \int_\Gamma |\nabla v|^2 + 
C_{\rm TF} \int_\Gamma |v|^{10/3} + \int_\Gamma  V^{\rm ion} |v|^2 + \frac
12 D_\Gamma(|v|^2,|v|^2).
$$

It is well known \cite{Lieb} that (\ref{eq:minTFWrho}) has a unique
minimizer $\rho^0$, and that the minimizers of 
  (\ref{eq:minTFWu}) are $u$ and $-u$, where
  $u=\sqrt{\rho^0}$. Besides, the function $u$ is in $
  H^{m+1/2-\epsilon}_\#(\Gamma)$ for any $\epsilon > 0$ (and therefore in
  $C^{2}_\#(\Gamma)$ since $m+1/2-\epsilon > 7/2$ for $\epsilon$ small
  enough), is positive everywhere in $\Gamma$ and
  satisfies the Euler equation 
$$
- \frac{C_{\rm W}}2 \Delta u + \left(\frac 53 C_{\rm TF} u^{4/3} +
V^{\rm ion}  + V_{u^2}^{\rm Coulomb} \right) u = \lambda u 
$$
for some $\lambda \in \R$, where 
$$
V_{\rho^0}^{\rm Coulomb}(x) = 4\pi
\sum_{k \in \cR^\ast \setminus \left\{0\right\}} |k|^{-2}
\widehat{\rho}_k e_k(x)
$$
is the periodic Coulomb potential generated by the periodic charge
distribution $\rho$. Recall that $V^{\rm Coulomb}_\rho$ can also be
defined as the unique solution in $H^1_\#(\Gamma)$ to 
$$
\left\{ \begin{array}{l}
\dps - \Delta V_\rho^{\rm Coulomb} = 4\pi \left(
  \rho-|\Gamma|^{-1}\int_\Gamma\rho \right) \\
\dps \int_\Gamma\rho V_\rho^{\rm Coulomb} = 0.
\end{array} \right.
$$
The planewave discretization of the TFW model is obtained by choosing 
\begin{enumerate}
\item an energy cut-off $\Ec > 0$ or, equivalently, a finite
  dimensional Fourier space $V_{N_c}$, the integer $N_c$ being related
  to $\Ec$ through the relation $N_c :=
[\sqrt{2\Ec} \, L/2\pi]$;
\item a cartesian grid $\cG_{N_g}$ with step size
$L/N_g$ where $N_g \in \N^\ast$ is such
that $N_g \ge 4N_c+1$, 
\end{enumerate}
and by considering the finite dimensional minimization problem 
\begin{equation} 
I^{\rm TFW}_{N_c,N_g}  =  \inf \left\{ E^{\rm TFW}_{N_g}(v_{N_c}),
  \; v_{N_c} \in V_{N_c},
  \; \int_\Gamma |v_{N_c}|^2 = \cN \right\}, \label{eq:minTFWuN} 
\end{equation}
where
\begin{eqnarray*} 
E^{\rm TFW}_{N_g}(v_{N_c}) & = & \frac{C_{\rm W}}2 \int_\Gamma
|\nabla v_{N_c}|^2 +  
C_{\rm TF} \int_\Gamma \cI_{N_g}(|v_{N_c}|^{10/3}) + \int_\Gamma
\cI_{N_g}(V^{\rm ion}) |v_{N_c}|^2 \\ && + \frac 12
D_\Gamma(|v_{N_c}|^2,|v_{N_c}|^2),
\end{eqnarray*} 
$\cI_{N_g}$ denoting the interpolation operator introduced in the 
previous section. The Euler equation associated with
(\ref{eq:minTFWuN}) can be written as a nonlinear eigenvalue problem
$$
\forall v_{N_c} \in V_{N_c}, \quad \langle (\widetilde
H_{|u_{N_c,N_g}|^2}^{N_g} u_{N_c,N_g} - \lambda_{N_c,N_g})
u_{N_c,N_g},v_{N_c} \rangle_{H^{-1}_\#,H^1_\#} = 0
$$
where we have denoted by
$$
\widetilde H_{\rho}^{N_g} =  - \frac{C_{\rm W}}2 \Delta +
\cI_{N_g}\left(\frac 53 C_{\rm TF} \rho^{2/3} + V^{\rm ion}\right) +
V_{\rho}^{\rm Coulomb}    
$$
the pseudo-spectral TFW Hamiltonian associated with the density $\rho$,
and by $\lambda_{N_c,N_g}$ the Lagrange multiplier of the constraint
$\int_\Gamma |v_{N_c}|^2 = \cN$.
We therefore have
$$
- \frac{C_{\rm W}}2 \Delta u_{N_c,N_g}+ \Pi_{N_c} \left[
\left(\cI_{N_g}\left( \frac 53 C_{\rm TF}|u_{N_c,N_g}|^{4/3} + V^{\rm ion}\right) + V_{|u_{N_c,N_g}|^2}^{\rm Coulomb}\right)u_{N_c,N_g}
\right]  = \lambda_{N_c,N_g} u_{N_c,N_g}.
$$
Under the condition that $N_g \ge 4N_c+1$, we have for all $\phi \in C^0_\#(\Gamma)$,  
$$
\forall (k,l) \in \cR^\ast \times \cR^\ast \mbox{ s.t. } |k|
,|l| \le \frac{2\pi}L N_c, \quad \int_\Gamma \cI_{N_g}(\phi) \, 
e_k^\ast \, e_l  = \widehat{\phi}^{{\rm FFT}}_{k-l},
$$
so that,  $\widetilde H_{u_{N_c,N_g}}$ is defined on $V_{N_c}$  by
the Fourier matrix 
\begin{eqnarray*}
[\widehat H_{|u_{N_c,N_g}|^2}^{N_g}]_{kl} &=& \frac{C_{\rm W}}2 |k|^2 \delta_{kl} + 
\frac 53 C_{\rm TF} \widehat{(|u_{N_c,N_g}|^{4/3})}_{k-l}^{{\rm FFT},N_g}
+ \widehat{(V^{\rm ion})}_{k-l}^{{\rm FFT},N_g} \\ &&  
+4\pi \frac{\widehat{(|u_{N_c,N_g}|^2)}_{k-l}^{{\rm FFT},N_g}}{|k-l|^2}
\left( 1 - \delta_{kl} \right),
\end{eqnarray*}
where, by convention, the last term of the right hand side is equal to
zero for $k=l$.

\medskip

\noindent
We also introduce the variational approximation of (\ref{eq:minTFWu}) 
\begin{equation}
I^{\rm TFW}_{N_c}  =  \inf \left\{ E^{\rm TFW}(v_{N_c}),
  \; v_{N_c} \in V_{N_c},
  \; \int_\Gamma |v_{N_c}|^2 = \cN \right\}. \label{eq:minTFWuV} 
\end{equation}
Any minimizer $u_{N_c}$ to (\ref{eq:minTFWuV}) satisfies the
elliptic equation
\begin{equation} \label{eq:Euler_VNc}
- \frac{C_{\rm W}}2 \Delta u_{N_c} + \Pi_{N_c} \left[ 
\frac 53 C_{\rm TF} |u_{N_c}|^{4/3} u_{N_c}+ V^{\rm ion}u_{N_c} + V_{|u_{N_c}|^2}^{\rm Coulomb} u_{N_c} 
 \right] = \lambda_{N_c}u_{N_c},
\end{equation}
for some $\lambda_{N_c} \in \R$.

\medskip

\begin{theorem} \label{Th:TFW} For each $N_c \in \N$, we denote by
  $u_{N_c}$ a minimizer to (\ref{eq:minTFWuV}) such that  
  $(u_{N_c},u)_{L^2_\#} \ge 0$ and, for each $N_c
  \in \N$ and $N_g \ge 4N_c+1$, we 
  denote by $u_{N_c,N_g}$ a minimizer to (\ref{eq:minTFWuN}) such that 
  $(u_{N_c,N_g},u)_{L^2_\#} \ge 0$. Then for $N_c$ large enough, $u_{N_c}$
  and $u_{N_c,N_g}$ are unique, and the following estimates hold true
\begin{eqnarray}
\|u_{N_c}-u\|_{H^s_\#} &\le& C_s N_c^{-(m-s+1/2-\epsilon)}
\label{eq:estim_Nc_u} \\
|\lambda_{N_c}-\lambda| &\le & C N_c^{-(2m-1-\epsilon)}
\label{eq:estim_Nc_lambda} \\
\gamma \|u_{N_c}-u\|_{H^1_\#}^2 \le I^{\rm TFW}_{N_c} -I^{\rm TFW}
&\le& C \|u_{N_c}-u\|_{H^1_\#}^2
\label{eq:estim_Nc_I}
\\
\|u_{N_c,N_g}-u_{N_c}\|_{H^s_\#} &\le & C_s \, 
N_c^{3/2+(s-1)_+} N_g^{-m}, \label{eq:estim_NcNg_u}  \\ 
|\lambda_{N_c,N_g}-\lambda_{N_c}| &\le& C 
  N_c^{3/2}N_g^{-m},\label{eq:estim_NcNg_lambda} \\
|I^{\rm TFW}_{N_c,N_g} -I^{\rm TFW}_{N_c}| &\le & C
N_c^{3/2}N_g^{-m} , \label{eq:estim_NcNg_I}
\end{eqnarray}
for all $- m +3/2 < s < m+1/2$ and for some constants $\gamma > 0$, $C
\ge 0$ and $C_s \ge 0$ independent of $N_c$ and $N_g$.
\end{theorem}

\medskip

\begin{remark} \label{Th:OF}
More complex orbital-free models have been proposed in the recent years
\cite{orbitalfree}, which are used to perform
multimillion atom DFT calculations. Some of these models however are not
well posed  (the energy functional is not bounded below 
\cite{BlancCances}), and 
the others are not well understood from a mathematical point of
view. For these reasons, we will not deal with those models in this
article.
\end{remark}

\begin{proof}[of Theorem~\ref{Th:TFW}]
The estimates (\ref{eq:estim_Nc_u}), (\ref{eq:estim_Nc_lambda}) and
(\ref{eq:estim_Nc_I}) originate from arguments already introduced in
\cite{CCM}. For brevity, we only recall the main steps of the
proof and leave the details to the reader. 

\medskip

\noindent
The difference between (\ref{eq:minTFWu}) and the problem dealt
with in \cite{CCM} is the presence of the Coulomb term
$D_\Gamma(|v|^2,|v|^2)$, for which the following estimates are
available: 
\begin{eqnarray}
0 \le D_\Gamma(\rho,\rho) & \le & C \|\rho\|_{L^2_\#}^2, \quad \mbox{
  for all } \rho \in L^2_\#(\Gamma),  \label{eq:estim_D1} \\
|D_\Gamma(uv,uw)| & \le  & C \|v\|_{L^2_\#} \|w\|_{L^2_\#},  \quad \mbox{
  for all } (v,w) \in (L^2_\#(\Gamma))^2, \label{eq:estim_D2} \\
|D_\Gamma(\rho,vw)| & \le & C \|\rho\|_{L^2_\#} \|v\|_{L^2_\#}
\|w\|_{L^2_\#}, \quad \mbox{
  for all } (\rho,v,w) \in (L^2_\#(\Gamma))^3, \qquad
\label{eq:estim_D3} \\
\|V_\rho^{\rm Coulomb}\|_{L^\infty} &\le& C \|\rho\|_{L^2_\#} , \quad \mbox{
  for all } \rho \in L^2_\#(\Gamma) \label{eq:estim_D4} \\
\|V_\rho^{\rm Coulomb}\|_{H^{s+2}_\#} &\le& C \|\rho\|_{H^s_\#}, \quad \mbox{
  for all } \rho \in H^s_\#(\Gamma).\label{eq:estim_D5}
\end{eqnarray}
Here and in the sequel, $C$ denotes a non-negative constant which may
depend on $\Gamma$, $V^{\rm ion}$ and $\cN$, but not on the
discretization parameters.

Let $F(t)=C_{\rm TF}t^{5/3}$ and $f(t) = F'(t) = \frac 53 C_{\rm
  TF}t^{2/3}$. The function $F$ is in $C^1([0,+\infty))\cap 
C^\infty((0,+\infty))$, is strictly convex on $[0,+\infty)$, and
for all $(t_1,t_2) \in \R_+ \times \R_+$,
\begin{equation} \label{eq:fff}
|f(t_2^2)t_2- f(t_1^2)t_2 - 2f'(t_1^2)t_1^2(t_2-t_1)| 
\le \frac{70}{27} C_{\rm TF} \max(t_1^{1/3},t_2^{1/3}) \, |t_2-t_1|^2.
\end{equation}
The first and second derivatives of
$E^{\rm TFW}$ at the unique positive minimizer $u=\sqrt{\rho^0}$ of
(\ref{eq:minTFWu}) are respectively given by
\begin{eqnarray*}
&& \langle {E^{\rm TFW}}'(u),v \rangle_{H^{-1}_\#,H^1_\#} 
= 2 \langle H_{\rho^0}u,v \rangle \\
&& \langle {E^{\rm TFW}}''(u)v,w \rangle_{H^{-1}_\#,H^1_\#} 
= 2 \langle H_{\rho^0}v,w \rangle + 4D_\Gamma(uv,uw)+ 4 \int_\Gamma
f'(|u|^2) |u|^2 vw, 
\end{eqnarray*}
where we have denoted by 
$$
H_{\rho} = - \frac{C_{\rm W}}2 \Delta + f(\rho) + V^{\rm ion} +
V_{\rho}^{\rm Coulomb}   
$$
the TFW Hamiltonian associated with the density $\rho$. We recall (see  
\cite{Lieb} and the proof of Lemma~2 in \cite{CCM}) that (i) $u \in
H^{m+1/2-\epsilon}_\#(\Gamma) \cap C^{2}_\#(\Gamma)$ for each $\epsilon
> 0$, (ii) $u > 0$ on $\R^3$, (iii) $\lambda$ is the
ground state eigenvalue of $H_{\rho^0}$ and is non-degenerate.
Using (\ref{eq:estim_D1}), (\ref{eq:estim_D2}) and the fact that $f' >
0$ on $(0,+\infty)$, we can then show (see the proof of Lemma~1 in
\cite{CCM}) that there exist $\beta > 0$, $\gamma > 0$ and $M \ge 0$
such that for all $v \in H^1_\#(\Gamma)$,
\begin{eqnarray}
&& 
0 \le \langle (H_{\rho^0}-\lambda)v,v \rangle_{H^{-1}_\#,H^1_\#}  \le
M \|v\|_{H^1_\#}^2 
\\
&& \beta \|v\|_{H^1_\#}^2  \le 
\langle ({E^{\rm TFW}}''(u)-2\lambda)v,v \rangle_{H^{-1}_\#,H^1_\#}
\le M \|v\|_{H^1_\#}^2, \label{eq:NRJsecondContinue}
\end{eqnarray}
and for all $v \in H^1_\#(\Gamma)$ such that $\|v\|_{L^2_\#}=\cN^{1/2}$ and
$(v,u)_{L^2_\#} \ge 0$,
\begin{equation} \label{eq:borne_1}
\gamma \|v-u\|_{H^1_\#}^2 \le 
\langle (H_{\rho^0}-\lambda)(v-u),(v-u) \rangle_{H^{-1}_\#,H^1_\#}. 
\end{equation}
Remarking that
\begin{eqnarray} \!\!\!\!\!\!\!\!\!\!\!\!\!
E^{\rm TFW}(u_{N_c})- E^{\rm TFW}(u) \!\!\! & = & \!\!\! \langle
(H_{\rho^0}-\lambda)(u_{N_c}-u),(u_{N_c}-u)\rangle_{H^{-1}_\#,H^1_\#}
\nonumber \\
\!\!\! && \!\!\! + \frac 12
D_\Gamma(|u_{N_c}|^2-|u|^2,|u_{N_c}|^2-|u|^2) \nonumber \\ \!\!\!
&& \!\!\! + \int_\Gamma F(|u_{N_c}|^2)-F(|u|^2)-f(|u|^2)(|u_{N_c}|^2-|u|^2)
\label{eq:erreur_TFW_Nc}
\end{eqnarray}
and using (\ref{eq:borne_1}), the positivity of the bilinear form
$D_\Gamma$, and the convexity of the function $F$, we obtain that 
$$
I^{\rm TFW}_{N_c}-I^{\rm TFW} = E^{\rm TFW}(u_{N_c})- E^{\rm TFW}(u) 
\ge \gamma \|u_{N_c}-u\|_{H^1_\#}^2. 
$$
For each $N_c \in \N$, 
$\widetilde u_{N_c}=\cN^{1/2}\Pi_{N_c}u/\|\Pi_{N_c}u\|_{L^2_\#}$ satisfies
$\widetilde u_{N_c}$ and $\|\widetilde u_{N_c}\|_{L^2_\#} =\cN^{1/2}$, and the
sequence $(\widetilde u_{N_c})_{N_c \in \N}$ converges to $u$ in
$H^{m+1/2-\epsilon}_\#(\Gamma)$ for each $\epsilon > 0$. As the
functional $E^{\rm TFW}$ is continuous on $H^1_\#(\Gamma)$, we have 
$$
 \|u_{N_c}-u\|_{H^1_\#}^2 \le \gamma^{-1} \left(I^{\rm TFW}_{N_c}-I^{\rm
     TFW}\right) 
\le \gamma^{-1} \left(E^{\rm TFW}(\widetilde u_{N_c}) - E^{\rm
    TFW}(u)\right) 
\mathop{\longrightarrow}_{N_c \to \infty} 0.
$$
Hence, $(u_{N_c})_{N_c \in \N}$ converges to $u$ in
$H^1_\#(\Gamma)$, and we also have
\begin{eqnarray*}
\lambda_{N_c} & = & \cN^{-1} \bigg[ \frac 12 \int_{\Gamma} |\nabla
u_{N_c}|^2
+ \int_\Gamma f(|u_{N_c}|^2) |u_{N_c}|^2
+ \int_\Gamma V^{\rm ion} |u_{N_c}|^2 + D_\Gamma(|u_{N_c}|^2,|u_{N_c}|^2)
 \bigg] \\
& \dps \mathop{\longrightarrow}_{N_c \to \infty} & \cN^{-1} \bigg[ 
\frac 12 \int_{\Gamma} |\nabla u|^2 + \int_\Gamma f(|u|^2) |u|^2 
+ \int_\Gamma V^{\rm ion} |u|^2 + D_\Gamma(|u|^2,|u|^2)
\bigg] \\
& = & \lambda.
\end{eqnarray*}
As $f(|u_{N_c}|^2)u_{N_c}+ V^{\rm ion}u_{N_c} + V_{|u_{N_c}|^2}^{\rm Coulomb} u_{N_c} 
$ is bounded in $L^2_\#(\Gamma)$, uniformly in
$N_c$, we deduce from (\ref{eq:Euler_VNc}) that the sequence
$(u_{N_c})_{N_c \in \N}$ is bounded in $H^2_\#(\Gamma)$, hence in
$L^\infty(\Gamma)$. Now 
\begin{eqnarray*}
\Delta (u_{N_c}-u) & = & 2C_{\rm W}^{-1} \bigg[ \Pi_{N_c} \bigg( 
f(|u_{N_c}|^2)u_{N_c} - f(|u|^2)u + V^{\rm ion}(u_{N_c}-u) + 
 \\
&& \qquad \qquad \qquad V_{|u_{N_c}|^2}^{\rm Coulomb}
u_{N_c}-V_{|u|^2}^{\rm Coulomb}u 
 \bigg) \\
&  & \qquad  \quad + \left(1-\Pi_{N_c}\right) \left(f(|u|^2)u + V^{\rm ion} u + V_{|u|^2}^{\rm
    Coulomb}u \right) \\
&& \qquad  \quad 
- \lambda_{N_c}(u_{N_c}-u) - (\lambda_{N_c}-\lambda) u \bigg].
\end{eqnarray*}
Observing that the right-hand side goes to zero in $L^2_\#(\Gamma)$ when
$N_c$ goes to infinity, we obtain that $(u_{N_c})_{N_c \in \N}$
converges to $u$ in $H^2_\#(\Gamma)$, and therefore in
$C^{0,1/2}_\#(\Gamma)$. By a simple bootstrap argument, one can see that
the convergence also holds in $H^{m+1/2-\epsilon}_\#(\Gamma)$ for each
$\epsilon > 0$. 
The upper bound in~(\ref{eq:estim_Nc_I}) is obtained from
(\ref{eq:erreur_TFW_Nc}), remarking that
\begin{eqnarray*}
0 & \le & \int_\Gamma
F(|u_{N_c}|^2)-F(|u|^2)-f(|u|^2)(|u_{N_c}|^2-|u|^2) \\
& \le & \frac{35}9 C_{\rm TF} \int_\Gamma \max(|u_{N_c}|^{4/3},|u|^{4/3})
|u_{N_c}-u|^2 \\
& \le & \frac{35}9 C_{\rm TF} \left(\max_{N_c \in \N}
  \|u_{N_c}\|_{L^\infty}\right)^{4/3} \, \|u_{N_c}-u\|_{L^2_\#}^2
\end{eqnarray*}
and that
\begin{eqnarray*}
0 & \le & D_\Gamma(|u_{N_c}|^2-|u|^2,|u_{N_c}|^2-|u|^2) \le  C
\||u_{N_c}|^2-|u|^2\|_{L^2_\#}^2 \\ 
& \le & 4 C  \, \left(\max_{N_c \in \N} \|u_{N_c}\|_{L^\infty}\right)^2
 \, \|u_{N_c}-u\|_{L^2_\#}^2.
\end{eqnarray*}

\medskip

\noindent
The uniqueness of $u_{N_c}$ for $N_c$ large enough can then be checked
as follows. First, $(u_{N_c},\lambda_{N_c})$ satisfies the variational equation
$$
\forall v_{N_c} \in V_{N_c}, \quad 
\langle
(H_{|u_{N_c}|^2}-\lambda_{N_c})u_{N_c},v_{N_c}\rangle_{H^{-1}_\#,H^1_\#}
= 0.  
$$
Therefore $\lambda_{N_c}$ is the variational approximation  in $V_{N_c}$
of some eigenvalue of $H_{|u_{N_c}|^2}$. As $(u_{N_c})_{N_c \in
  \N}$ converges to $u$ in $L^\infty(\Gamma)$,
$H_{|u_{N_c}|^2}-H_{\rho^0}$ converges to $0$ in operator
norm. Consequently, the $n^{\rm th}$ eigenvalue of $H_{|u_{N_c}|^2}$
converges to the $n^{\rm th}$ eigenvalue of $H_{\rho^0}$ when $N_c$ goes
to infinity, the convergence being uniform in $n$. Together
with the fact that the sequence $(\lambda_{N_c})_{N_c \in \N}$ converges
to $\lambda$, the non-degenerate ground state eigenvalue of
$H_{\rho^0}$, this implies that for $N_c$ large enough, $\lambda_{N_c}$
is the ground state eigenvalue of $H_{|u_{N_c}|^2}$ in $V_{N_c}$ and
for all $v_{N_c} \in V_{N_c}$ such that $\|v_{N_c}\|_{L^2_\#}=\cN^{1/2}$ and
$(v_{N_c},u_{N_c})_{L^2_\#} \ge 0$,
\begin{eqnarray}
E^{\rm TFW}(v_{N_c})- E^{\rm TFW}(u_{N_c}) & = & \langle
(H_{|u_{N_c}|^2}-\lambda_{N_c})(v_{N_c}-u_{N_c}),(v_{N_c}-u_{N_c})\rangle_{H^{-1}_\#,H^1_\#}  \nonumber
\\
&& + \frac 12 D_\Gamma(|v_{N_c}|^2-|u_{N_c}|^2,|v_{N_c}|^2-|u_{N_c}|^2)\nonumber \\
&& + \int_\Gamma
F(|v_{N_c}|^2)-F(|u_{N_c}|^2)-f(|u_{N_c}|^2)(|v_{N_c}|^2-|u_{N_c}|^2)
\nonumber \\
& \ge &  \langle
(H_{|u_{N_c}|^2}-\lambda_{N_c})(v_{N_c}-u_{N_c}),(v_{N_c}-u_{N_c})\rangle_{H^{-1}_\#,H^1_\#} \nonumber
\\
& \ge & \frac{\gamma}2 \|v_{N_c}-u_{N_c}\|_{H^1_\#}^2. \label{eq:pre_uniqueness}
\end{eqnarray}
It easily follows that for $N_c$ large enough, (\ref{eq:minTFWuV}) has a
unique minimizer $u_{N_c}$ such that $(u_{N_c},u)_{L^2_\#} \ge 0$.

\medskip

Let us now establish the rates of convergence of
$|\lambda_{N_c}-\lambda|$ and $\|u_{N_c}-u\|_{H^s_\#}$. First,
\begin{eqnarray*}
\lambda_{N_c}-\lambda & = &
\cN^{-1} \left[\langle (H_{|u|^2}-\lambda)(u_{N_c}-u),(u_{N_c}-u)\rangle_{H^{-1}_\#,H^1_\#}
+ \int_\Gamma w_{N_c} (u_{N_c}-u) \right]
\end{eqnarray*}
with
$$
w_{N_c}=\frac{f(|u_{N_c}|^2)-f(|u|^2)}{u_{N_c}-u}|u_{N_c}|^2
+ V^{\rm Coulomb}_{|u_{N_c}|^2}(u_{N_c}+u).
$$
We know that the sequence $(u_{N_c})_{N_c \in \N}$ converges to $u$ in
$H^{m+1/2-\epsilon}_\#(\Gamma)$ and that $u > 0$ in
$\R^3$. Consequently, for $N_c$ large enough,
the function $u_{N_c}$ (which is continuous and $\cR$-periodic) is
bounded away from $0$, uniformly in $N_c$. As $f \in
C^\infty((0,+\infty))$, the function $w_{N_c}$ is uniformly bounded in 
$H^{m-3/2-\epsilon}_\#(\Gamma)$ (at least for $N_c$ large enough). We
therefore obtain that for all $0 \le r < m-3/2$, there exists a constant
$C_r \in \R_+$ such that for all $N_c$ large enough,
\begin{equation} \label{eq:pre_estim_lambda}
|\lambda_{N_c}-\lambda| \le C_r \left( \|u_{N_c}-u\|_{H^1}^2
+ \| u_{N_c}-u \|_{H^{-r}} \right).
\end{equation}
In order to evaluate the $H^1_\#$-norm of the error $(u_{N_c}-u)$, we first
notice that 
\begin{equation}\label{eq:4:1:H1}
\forall v_{N_c} \in V_{N_c}, \quad 
\| u_{N_c}-u \|_{H^1_\#} \le \| u_{N_c} - v_{N_c}  \|_{H^1_\#} 
+\| v_{N_c}  - u \|_{H^1_\#},
\end{equation}
and that 
 \begin{eqnarray} \!\!\!\!\!\!\!\!\!\!\!\!
 \| u_{N_c}  - v_{N_c} \|_{H^1_\#}^2 & \le & \beta^{-1} \,  
\langle ({E^{\rm TFW}}''(u) - 2\lambda)(u_{N_c}  - v_{N_c} ),
(u_{N_c}  - v_{N_c} ) \rangle_{H^{-1}_\#,H^1_\#} \nonumber \\
& = & \beta^{-1}  \bigg( \langle ({E^{\rm TFW}}''(u) - 2\lambda)(u_{N_c}  - u),
(u_{N_c}  - v_{N_c} ) \rangle_{H^{-1}_\#,H^1_\#} \nonumber \\ 
&  & \quad  +\langle ({E^{\rm TFW}}''(u) - 2\lambda)(u- v_{N_c} ),(u_{N_c}  -
  v_{N_c} ) \rangle_{H^{-1}_\#,H^1_\#} \bigg). 
 \label{eq:4:3:H1}
 \end{eqnarray}
For all $z_{N_c} \in V_{N_c}$, 
\begin{eqnarray}
&& \!\!\!\!\!\!\!\!\!\!\!\!\!\!\!\!
\langle ({E^{\rm TFW}}''(u)-2\lambda)(u_{N_c}-u),z_{N_c}
\rangle_{H^{-1}_\#,H^1_\#} \nonumber\\ & = & 
- 2 \int_\Gamma [f(|u_{N_c}|^2)u_{N_c}-f(|u|^2)u - 2 f'(|u|^2)|u|^2
(u_{N_c}-u)] z_{N_c} \nonumber \\
& & - 2D_\Gamma((u_{N_c}-u)(u_{N_c}+u),(u_{N_c}-u)z_{N_c})
- 2D_\Gamma((u_{N_c}-u)^2,uz_{N_c}) \nonumber \\ &&
+ 2 (\lambda_{N_c}-\lambda) \int_\Gamma u_{N_c}z_{N_c}. \label{eq:new36}
\end{eqnarray}
On the other hand, we have for all $v_{N_c} \in V_{N_c}$ such that
$\|v_{N_c}\|_{L^2_\#}={\mathcal N}^{1/2}$, 
$$
\int_\Gamma u_{N_c} (u_{N_c}-v_{N_c}) =\cN - \int_\Gamma  u_{N_c}
v_{N_c}
= \frac 12 \|u_{N_c}-v_{N_c}\|_{L^2_\#}^2.
$$
Using (\ref{eq:estim_D3}), (\ref{eq:fff}), (\ref{eq:pre_estim_lambda})
with $r=0$ and the above
equality, we therefore obtain for all $v_{N_c} \in V_{N_c}$ such
 that $\|v_{N_c}\|_{L^2_\#}=\cN^{1/2}$,
 \begin{eqnarray}
&& 
\!\!\!\!\!\!\!\!\!\!\!\!\!
\left| \langle ({E^{\rm TFW}}''(u) - 2\lambda)(u_{N_c}-u),(u_{N_c}-v_{N_c})
  \rangle_{H^{-1}_\#,H^1_\#}\right|
\nonumber \\
&& \le  C \bigg( \|u_{N_c}-u\|^2_{H^1_\#} \, \| u_{N_c}-v_{N_c}
\|_{H^1_\#}  \nonumber \\
&& \qquad  \quad + 
\left( \|u_{N_c}-u\|_{H^1_\#}^2 + \|u_{N_c}-u\|_{L^{2}_\#} \right) 
\| u_{N_c}-v_{N_c} \|_{L^2_\#}^2  \bigg). \label{eq:4:3:H2}
 \end{eqnarray}
Therefore, for $N_c$ large enough, we have for all
$v_{N_c} \in V_{N_c}$ such that $\|v_{N_c}\|_{L^2_\#}=\cN^{1/2}$,  
$$
 \| u_{N_c}  - v_{N_c} \|_{H^1_\#} \le   C 
 \left(   \|u_{N_c}-u\|_{H^1_\#}^2  + \|v_{N_c}-u \|_{H^1_\#} \right) .
$$
Together with (\ref{eq:4:1:H1}), this shows that
there exists $N \in \N$ and $C \in \R_+$ such that for all ${N_c} \ge N$, 
$$
\forall v_{N_c} \in V_{N_c} \mbox{ s.t. } \|v_{N_c}\|_{L^2_\#}=\cN^{1/2},
\quad \|u_{N_c} -u \|_{H^1_\#} \le C  \|v_{N_c}-u \|_{H^1_\#}.   
$$
By a classical argument (see e.g. the proof of Theorem~1 in \cite{CCM}),
we deduce from (\ref{eq:app-Fourier}) and the above inequality that
\begin{equation} \label{eq:estim_H1}
\|u_{N_c} -u \|_{H^1_\#} \le C \min_{v_{N_c} \in V_{N_c}} \|v_{N_c}-u
\|_{H^1_\#} \le C N^{-(m-1/2-\epsilon)},  
\end{equation}
for some constants $C$ independent of $N_c$.

\medskip

\noindent
For $w \in L^{2}_\#(\Gamma)$, we denote by $\psi_w$ the unique solution to the
adjoint problem 
\begin{equation} \label{eq:adjoint}
\left\{ \begin{array}{l}
\mbox{find } \psi_w \in u^\perp \mbox{ such that} \\
\forall v \in u^\perp, \quad \langle ({E^{\rm TFW}}''(u) -
2\lambda)\psi_w,v \rangle_{H^{-1}_\#,H^1_\#} = 
\langle w,v \rangle_{H^{-1}_\#,H^1_\#},
\end{array} \right.
\end{equation}
where
$$
u^\perp = \left\{ v \in H^1_\#(\Gamma) \; | \; \int_\Gamma uv=0 \right\}.
$$
The function $\psi_w$ is solution to the elliptic equation
\begin{eqnarray*} \!\!\!\!\!\!\!\!\!\!\!\!\!\!\!\!
&& - \frac{C_{\rm W}}2 \Delta \psi_w + \left( V^{\rm
    ion}+V_{u^2}^{\rm Coulomb}+f(u^2)+2f'(u^2)u^2-\lambda \right) \psi_w 
+ 2 V_{u\psi_w}^{\rm Coulomb} u \nonumber \\ & & \qquad\qquad =
2 \left(\int_\Gamma f'(u^2)u^3\psi_w+D_{\Gamma}(u^2,u\psi_w)\right) u +
w - (w,u)_{L^2_\#} u, 
\end{eqnarray*}
from which we deduce that if $w \in H^r_\#(\Gamma)$ for some $0 \le r <
m-3/2$, then $\psi_w \in H^{r+2}_\#(\Gamma)$ and  
\begin{equation} \label{eq:borne_psiw}
\| \psi_w \|_{H^{r+2}_\#} \le C_r \| w \|_{H^{r}_\#}, 
\end{equation}
for some constant $C_r$ independent of $w$.
Let $u_{N_c}^*$ be the orthogonal projection, for the $L^2_\#$ inner
product, of  
$u_{N_c}$ on the affine space $\left\{v \in L^2_\#(\Gamma) \, | \,
  \int_\Gamma uv=\cN \right\}$. One has 
$$
u_{N_c} ^* \in H^1_{\#}(\Gamma), \qquad u_{N_c} ^*-u \in u^\perp,  \qquad 
u_{N_c}^*-u_{N_c}  = \frac 1{2\cN} \|u_{N_c}-u \|_{L^2_\#}^2 u, 
$$
from which we infer that
\begin{eqnarray*}
\|u_{N_c} -u\|_{L^2_\#}^2 & = & \int_\Gamma (u_{N_c}-u )(u_{N_c} ^*-u) + 
\int_\Gamma (u_{N_c}-u)(u_{N_c} - u_{N_c} ^*) \\
& = &  \int_\Gamma (u_{N_c}-u)(u_{N_c}^*-u) - \frac 1{2\cN}
\|u_{N_c}-u\|_{L^2_\#}^2 \int_\Gamma (u_{N_c}-u) u  \\
& = &  \int_\Gamma (u_{N_c}-u)(u_{N_c}^*-u) + \frac 1{2\cN}
\|u_{N_c}-u\|_{L^2_\#}^2 \left( \cN - \int_\Gamma u_{N_c} u \right) \\
& = &  \int_\Gamma (u_{N_c}-u)(u_{N_c}^*-u) + \frac 1{4\cN}
\|u_{N_c}-u\|_{L^2_\#}^4 \\ 
& = &   \langle u_{N_c}-u  ,u_{N_c}^*-u \rangle_{H^{-1}_\#,H^1_\#}  + \frac 1{4\cN}
\|u_{N_c}-u\|_{L^2_\#}^4  \\ 
& = &   \langle ({E^{\rm TFW}}''(u)-2\lambda) \psi_{u_{N_c}-u}, u_{N_c}^*-u
\rangle_{H^{-1}_\#,H^1_\#} + \frac 1{4\cN}
\|u_{N_c}-u\|_{L^2_\#}^4 \\
& = &   \langle  ({E^{\rm TFW}}''-2\lambda)  (u_{N_c}-u), \psi_{u_{N_c}-u}
\rangle_{H^{-1}_\#,H^1_\#} +  \frac 1{4\cN}
\|u_{N_c}-u\|_{L^2_\#}^4 \\ 
& & + \frac 1{2\cN} \|u_{N_c}-u\|_{L^2_\#}^2 
 \langle  ({E^{\rm TFW}}''(u)-2\lambda) u , \psi_{u_{N_c}-u} \rangle_{H^{-1}_\#,H^1_\#} 
 \\ 
& = &   \langle  ({E^{\rm TFW}}''(u)-2\lambda)  (u_{N_c}-u), \psi_{u_{N_c}-u}
\rangle_{H^{-1}_\#,H^1_\#} + \frac 1{4\cN} \|u_{N_c}-u \|_{L^2_\#}^4 \\ 
& & +  \frac{2}{\cN} \|u_{N_c}-u\|_{L^2_\#}^2 \left[\int_\Gamma f'(u^2) u^3
  \psi_{u_{N_c}-u} + D_\Gamma(u^2,u \psi_{u_{N_c}-u}) \right] .
\end{eqnarray*}
For all $\psi_{N_c} \in V_{N_c}$, it therefore holds
\begin{eqnarray}
\|u_{N_c}-u\|_{L^2}^2 & = &   
\langle  ({E^{\rm TFW}}''(u)-2\lambda)  (u_{N_c}-u), \psi_{u_{N_c}-u}-\psi_{N_c}
\rangle_{H^{-1}_\#,H^1_\#} \nonumber \\
& & + \langle  ({E^{\rm TFW}}''(u)-2\lambda)  (u_{N_c}-u),
\psi_{N_c} \rangle_{H^{-1}_\#,H^1_\#} + \frac 1{4\cN} \|u_{N_c}-u
\|_{L^2_\#}^4  \nonumber  \\ 
& & +  \frac{2}{\cN} \|u_{N_c}-u\|_{L^2_\#}^2 \left[\int_\Gamma f'(u^2) u^3
  \psi_{u_{N_c}-u} + D_\Gamma(u^2,u \psi_{u_{N_c}-u}) \right] .
\label{eq:eeeee}
\end{eqnarray}
Using  (\ref{eq:estim_D3}), (\ref{eq:fff}), (\ref{eq:pre_estim_lambda})
with $r=0$ and (\ref{eq:new36}), we obtain that for all $\psi_{N_c} \in V_{N_c}
\cap u^\perp$,
\begin{eqnarray} \!\!\!\!\!\!\!\!\!\!\!\!\!\!\!\!\!\!   
\left| \langle ({E^{\rm TFW}}(u) -2 \lambda)(u_{N_c}-u),\psi_{N_c}
  \rangle_{H^{-1}_\#,H^1_\#}\right|
& \le & C \bigg(  \|u_{N_c}-u\|_{H^{1}_\#}^2 
\nonumber \\ & &
\!\!\!\!\!\!\!\!\!\!\!\!\!\!\!\!\!\!\!\!\!\!\!\!\!\!\!\!\!\!\!\!\!\!\!\!\!\!\!\!\!\!\!\!\!\!\!\!\!\!\!\!\!\!\!\!\!\!\!\!\!\!\!\!\!\!\!\!\!\!\!\!\!\!\!\!\!\!\!\!\!\!\!\!\!\!\!\!\!\!\!\!\!\!\!\!\!\!\!\!\!\!
+ \| u_{N_c}-u \|_{L^{2}_\#} 
\left( \|u_{N_c}-u\|_{H^1_\#}^2 + \|u_{N_c}-u\|_{L^{2}_\#} \right) 
 \bigg) \|\psi_{N_c}\|_{H^1_\#}.  \label{eq:new37}
 \end{eqnarray}
Let us denote by $\Pi^1_{V_{N_c} \cap u^\perp}$ the orthogonal projector
on $V_{N_c} \cap u^\perp$ for the $H^1_\#$ inner product and by 
$\psi_{N_c}^0 = \Pi^1_{V_{N_c} \cap u^\perp} \psi_{u_{N_c}-u}$.
Noticing that 
$$
\|\psi_{N_c}^0\|_{H^1_\#} \le  \|\psi_{u_{N_c}-u}\|_{H^1_\#}
 \le \beta^{-1} M\|u_{N_c}-u\|_{L^2_\#},
$$
we obtain from (\ref{eq:NRJsecondContinue}), (\ref{eq:eeeee}) and
(\ref{eq:new37}) that there exists $N \in \N$ and $C \in 
\R_+$ such that for all ${N_c} \ge N$,
$$
\|u_{N_c}-u\|_{L^2_\#}^2 \le  C \, \bigg(
  \|u_{N_c}-u\|_{L^2_\#} \|u_{N_c}-u\|_{H^{1}_\#}^2 
+   \|u_{N_c}-u\|_{H^1_\#}
\|\psi_{u_{N_c}-u}-\psi_{N_c}^0\|_{H^1_\#} \bigg). 
$$
Lastly, for all $v \in u^\perp$ and all $N_c \in \N^\ast$ 
\begin{equation} \label{eq:minuperp}
\|v - \Pi^1_{V_{N_c} \cap u^\perp} v \|_{H^1_\#}
\le \left(1+\frac{\cN^{1/2}}{2\pi L^{1/2} N_c \int_\Gamma u} \right)  
\|v- \Pi_{N_c} v \|_{H^1_\#},
\end{equation}
so that, in view of (\ref{eq:app-Fourier}) and (\ref{eq:borne_psiw}) 
\begin{eqnarray*}
\|\psi_{u_{N_c}-u}-\psi_{N_c}^0\|_{H^1_\#} &\le& C  \|\psi_{u_{N_c}-u}-
\Pi_{N_c} \psi_{u_{N_c}-u} \|_{H^1_\#} \\ &\le& C N_c^{-1}
\|\psi_{u_{N_c}-u} \|_{H^2_\#} \\ &\le &C N_c^{-1} \|u_{N_c}-u\|_{L^2_\#}.
\end{eqnarray*}
Therefore, 
\begin{eqnarray*}
\|u_{N_c}-u\|_{L^2_\#} & \le &  C \,   \bigg(
 \|u_{N_c}-u\|_{H^{1}_\#}^2  +  N_c^{-1} \|u_{N_c}-u\|_{H^1_\#}  \bigg)
 \\
& \le & C N_c^{-(m+1/2-\epsilon)}.
\end{eqnarray*}
By means of the inverse inequality
\begin{equation} \label{eq:inv_ineq}
\forall v_{N_c} \in V_{N_c}, \quad 
\|v_{N_c}\|_{H^r_\#} \le \left( \frac{2\pi}L \right)^{(r-s)}  {N_c}^{r-s} \|v_{N_c}\|_{H^s_\#}, 
\end{equation}
which holds true for all $s \le r$ and all ${N_c} \ge 1$,
we obtain that
\begin{equation}  \label{eq:error_H1_F}
\|u_{N_c}-u\|_{H^s_\#}  \le  C_s N_c^{-(m-s+1/2-\epsilon)} \qquad 
\mbox{for all } 0 \le s < m+1/2.
\end{equation}
To complete the first part of the proof of Theorem~\ref{Th:TFW}, we
still have to compute the $H^{-r}_\#$-norm of the error $(u_{N_c}-u)$
for $0 < r < m-3/2$. Let $w \in H^{r}_\#(\Gamma)$. Proceeding as above
we obtain   
\begin{eqnarray}
 \int_\Gamma w (u_{N_c}-u) & = &
\langle  ({E^{\rm TFW}}''(u)-2\lambda)(u_{N_c}-u), \Pi^1_{V_{N_c} \cap
  u^\perp} \psi_{w} 
\rangle_{H^{-1}_\#,H^1_\#}  \nonumber \\
& & + \langle  ({E^{\rm
    TFW}}''(u)-2\lambda)(u_{N_c}-u),\psi_{w}-\Pi^1_{V_{N_c} 
  \cap u^\perp}\psi_{w} 
\rangle_{H^{-1}_\#,H^1_\#} \nonumber \\ 
& &+ \frac{2}{\cN} \|u_{N_c}-u\|_{L^2_\#}^2\left[ \int_\Gamma f'(u^2) u^3
  \psi_{w}+ D_\Gamma(u^2,u\psi_w) \right] 
\nonumber\\&&\qquad- \frac 1{2\cN} \|u_{N_c}-u \|_{L^2_\#}^2 \int_\Gamma uw.
\label{eq:intom} 
\end{eqnarray}
Combining
 (\ref{eq:NRJsecondContinue}), (\ref{eq:borne_psiw}), (\ref{eq:new37}), 
(\ref{eq:minuperp}), (\ref{eq:error_H1_F}) and (\ref{eq:intom}), we
obtain that there exists a constant $C \in \R_+$ such that for all
${N_c}$ large enough and all $w \in H^{r}_\#( \Gamma)$,
\begin{eqnarray*}
\int_ \Gamma w (u_{N_c}-u) & \le & C' \left( \|u_{N_c}-u\|_{H^1_\#}^2 
+ {N_c}^{-(r+1)} \|u_{N_c}-u\|_{H^1_\#} \right) \|w\|_{H^r_\#} \\
 & \le & C \, N_c^{-(m+r+1/2-\epsilon)} \|w\|_{H^{r}_\#} .
\end{eqnarray*}

Therefore
\begin{equation} \label{eq:Hm1boundFourier}
\|u_{N_c}-u\|_{H^{-r}_\#} = \sup_{w \in H^{r}_\#( \Gamma) \setminus
  \left\{0\right\}} \frac{\dps \int_ \Gamma w (u_{N_c}-u)}{\|w\|_{H^{r}_\#}}
\le C \, N_c^{-(m+r+1/2-\epsilon)} ,
\end{equation}
for some constant $C \in \R_+$ independent of ${N_c}$.
Using (\ref{eq:pre_estim_lambda}), (\ref{eq:estim_H1}) and (\ref{eq:Hm1boundFourier}), we end up with 
$$
|\lambda_N-\lambda| \le  C N_c^{-(2m-1-\epsilon)} .
$$

\medskip

Let us now turn to the pseudospectral approximation (\ref{eq:minTFWuN})
of (\ref{eq:minTFWu}). First, we notice that
\begin{eqnarray*}
\frac{C_{\rm W}}2 \|\nabla u_{N_c,N_g}\|_{L^2_\#}^2 - \|V^{\rm
  ion}\|_{L^\infty} {\mathcal N} &\le & E^{\rm TFW}_{N_g}(u_{N_c,N_g})
\\ & \le & E^{\rm TFW}_{N_g}({\mathcal N}^{1/2}|\Gamma|^{-1/2}) \\
& \le &C_{\rm TF} {\mathcal N}^{5/3} |\Gamma|^{-2/3} + \|V^{\rm
  ion}\|_{L^\infty} {\mathcal N},
\end{eqnarray*}
from which we infer that $u_{N,N_g}$ is uniformly bounded in
$H^1_\#(\Gamma)$. We then see that
\begin{eqnarray*}
\lambda_{N_c,N_g} &=& \cN^{-1} \bigg[
\frac{C_{\rm W}}2 \int_\Gamma |\nabla u_{N_c,N_g}|^2 + 
\int_\Gamma \cI_{N_g}( V^{\rm
  ion}|u_{N_c,N_g}|^2+f(|u_{N_c,N_g}|^2)|u_{N_c,N_g}|^2)  \\ && \qquad\qquad
+D_\Gamma(|u_{N_c,N_g}|^2,|u_{N_c,N_g}|^2) \bigg].
\end{eqnarray*}
Using (\ref{eq:bound_INg1}), (\ref{eq:bound_INg2}) and
(\ref{eq:estim_D1}), we obtain that $\lambda_{N,N_c}$ 
also is uniformly bounded. Now,
\begin{eqnarray}
\Delta u_{N_c,N_g} & = & 2 C_{\rm W}^{-1} \Pi_{N_c}
\left(\cI_{N_g}\left(f(|u_{N_c,N_g}|^2)u_{N_c,N_g}\right) \right) + 
2 C_{\rm W}^{-1} \Pi_{N_c} \left( \cI_{N_g}\left(
  V^{\rm ion}u_{N_c,N_g} \right) \right) 
\nonumber \\ && + 2 C_{\rm W}^{-1}
\Pi_{N_c} \left(  V^{\rm Coulomb}_{|u_{N_c,N_g}|^2}  u_{N_c,N_g} \right) - 2 C_{\rm W}^{-1} 
\lambda_{N_c,N_g} u_{N_c,N_g}, \label{eq:EDP_uNcNg}
\end{eqnarray}
and we deduce from (\ref{eq:exact_integration}), (\ref{eq:bound_INg1})
and (\ref{eq:ineg_INg_4}) that 
\begin{eqnarray*}
\left\| \Pi_{N_c} \left( \cI_{N_g}\left( f(|u_{N_c,N_g}|^2)u_{N_c,N_g}
    \right) \right) \right\|_{L^2_\#}
& \le & \left( \int_\Gamma
  \left(\cI_{N_g}(f(|u_{N_c,N_g}|^2))\right)^2 |u_{N_c,N_g}|^2 \right)^{1/2} \\
& = & \left( \sum_{x \in \cG_{N_g} \cap \Gamma}
 \left( \frac{L}{N_g} \right)^3
  f(|u_{N_c,N_g}(x)|^2)^2|u_{N_c,N_g}(x)|^2\right)^{1/2} \\
& \le & \frac 53 C_{\rm TF} \|u_{N_c,N_g}\|_{L^\infty}^{1/3} 
   \left( \sum_{x \in \cG_{N_g}  \cap \Gamma}
\left( \frac{L}{N_g} \right)^3 |u_{N_c,N_g}(x)|^4 \right)^{1/2} \\
& = & \frac 53 C_{\rm TF} \|u_{N_c,N_g}\|_{L^\infty}^{1/3}
\|u_{N_c,N_g}\|_{L^4_\#}^2, 
\end{eqnarray*}
and that
\begin{eqnarray*}
\|\Pi_{N_c} \left( \cI_{N_g}\left(
  V^{\rm ion}u_{N_c,N_g} \right) \right) \|_{L^2_\#} & \le & 
\|\Pi_{2N_c} \left( \cI_{N_g}\left(
  V^{\rm ion}u_{N_c,N_g} \right) \right) \|_{L^2_\#} \\
& \le & \left( \int_\Gamma \cI_{N_g}(|V^{\rm ion}|^2|u_{N_c,N_g}|^2)
\right)^{1/2} \\
& \le & \|V^{\rm ion}\|_{L^\infty} \cN^{1/2}.
\end{eqnarray*}
Besides, using (\ref{eq:estim_D4}), 
\begin{eqnarray*}
\|\Pi_{N_c} \left( V^{\rm Coulomb}_{|u_{N_c,N_g}|^2} u_{N_c,N_g} \right)\|_{L^2_\#} & \le &
\|V^{\rm Coulomb}_{|u_{N_c,N_g}|^2} u_{N_c,N_g} \|_{L^2_\#} \\
 & \le & \cN^{1/2} \|V^{\rm Coulomb}_{|u_{N_c,N_g}|^2} \|_{L^\infty} \\
& \le & \cN^{1/2} \|u_{N_c,N_g}\|_{L^4_\#}^2.
\end{eqnarray*}
As $u_{N_c,N_g}$ is uniformly bounded in $H^1_\#(\Gamma)$, and therefore
in $L^4_\#(\Gamma)$, we get 
\begin{eqnarray*}
\|u_{N_c,N_g}\|_{H^2_\#} & = & \left( \|u_{N_c,N_g}\|_{L^2_\#}^2 +
  \|\Delta u_{N_c,N_g}\|_{L^2_\#}^2  \right)^{1/2} \\
& \le & C \left( 1 +   \|u_{N_c,N_g}\|_{L^\infty}^{1/3} \right) \\
& \le & C \left( 1 +   \|u_{N_c,N_g}\|_{H^2_\#}^{1/3} \right).
\end{eqnarray*}
Therefore $u_{N_c,N_g}$ is uniformly bounded in $H^2_\#(\Gamma)$, hence
in $L^\infty(\R^3)$.

Returning to (\ref{eq:EDP_uNcNg}) and using
(\ref{eq:ineg_INg_5}) and a bootstrap argument, we conclude that
$u_{N_c,N_g}$ is in fact uniformly bounded in $H^{7/2+\epsilon}_\#(\Gamma)$. 

Next, using (\ref{eq:pre_uniqueness}),
\begin{eqnarray*}
\frac\gamma 2 \|u_{N_c,N_g}-u_{N_c}\|_{H^1}^2 & \le & E^{\rm TFW}(u_{N_c,N_g})-E^{\rm TFW}(u_{N_c}) \\
& = & E^{\rm TFW}_{N_g}(u_{N_c,N_g})-E^{\rm TFW}_{N_g}(u_{N_c}) \\ && 
+ \int_\Gamma ((1-\cI_{N_g})(V))(|u_{N_c,N_g}|^2-|u_{N_c}|^2) \\ && 
+ \int_\Gamma (1-\cI_{N_g})(F(|u_{N_c,N_g}|^2)-F(|u_{N_c}|^2)) \\
& \le & \int_\Gamma ((1-\cI_{N_g})(V))(|u_{N_c,N_g}|^2-|u_{N_c}|^2) \\ && 
+ \int_\Gamma (1-\cI_{N_g})(F(|u_{N_c,N_g}|^2)-F(|u_{N_c}|^2)). 
\end{eqnarray*}
Let $g(t,t')=\frac{F(t'^2)-F(t^2)}{t'-t}$. For $N_c$ large enough,
$u_{N_c}$ is uniformly bounded away from zero; besides,
both $u_{N_c}$ and $u_{N_c,N_g}$ are uniformly bounded in
$H^{7/2+\epsilon}_\#(\Gamma)$. Therefore, $g(u_{N_c},u_{N_c,N_g})$ is uniformly bounded in
$H^{7/2+\epsilon}_\#(\Gamma)$. This implies that the Fourier coefficients of
$g(u_{N_c},u_{N_c,N_g})$ go to zero faster that $|k|^{-7/2}$, which
implies, using (\ref{eq:ineg_INg_2}) and (\ref{eq:ineg_INg_6}), that
\begin{eqnarray}
& & \left| \int_\Gamma (1-\cI_{N_g})(F(|u_{N_c,N_g}|^2)-F(|u_{N_c}|^2))
\right| \nonumber \\
& & \qquad = 
\left|\int_\Gamma
  (1-\cI_{N_g})\left(g(u_{N_c},u_{N_c,N_g})\right) \,
  (u_{N_c,N_g}-u_{N_c}) 
\right| \nonumber \\
& & \qquad \le \left\|\Pi_{N_c} \left(
    (1-\cI_{N_g})\left(g(u_{N_c},u_{N_c,N_g})\right) \right)
\right\|_{L^2_\#} \|u_{N_c,N_g}-u_{N_c}\|_{L^2_\#}
 \nonumber \\
& & \qquad \le C N_c^{3/2} N_g^{-7/2} \|u_{N_c,N_g}-u_{N_c}\|_{L^2_\#}.
\label{eq:bound_CC1}
\end{eqnarray}
On the other hand,
\begin{eqnarray*}
&& \left| \int_\Gamma ((1-\cI_{N_g})(V))(|u_{N_c,N_g}|^2-|u_{N_c}|^2)
\right| \\ && \qquad \le 
\|\Pi_{2N_c} ((1-\cI_{N_g})(V))\|_{L^2_\#}
\|u_{N_c,N_g}+u_{N_c}\|_{L^\infty} \|u_{N_c,N_g}-u_{N_c}\|_{L^2_\#} \\
&& \qquad \le C N_c^{3/2} N_g^{-m} \|u_{N_c,N_g}-u_{N_c}\|_{L^2_\#}.
\end{eqnarray*}
Therefore, 
\begin{eqnarray}
 \|u_{N_c,N_g}-u_{N_c}\|_{H^1_\#} & \le & C N_c^{3/2} N_g^{-7/2}.
\label{eq:bound_CC2}
\end{eqnarray}
We then deduce from (\ref{eq:bound_CC2}) and the inverse inequality
(\ref{eq:inv_ineq}) that $(u_{N_c,N_g})_{N_c,N_g \ge 4N_c+1}$ converges to $u$ in
$H^2_\#(\Gamma)$, and therefore in $L^\infty(\R^3)$. It follows that for
$N_c$ large enough, $u_{N_c,N_g}$ is bounded away from zero, which,
together with (\ref{eq:EDP_uNcNg}), implies that $(u_{N_c,N_g})_{N_c,N_g
  \ge 4N_c+1}$  is bounded in $H^{m+1/2-\epsilon}_\#(\Gamma)$. The
estimates (\ref{eq:bound_CC1}) and (\ref{eq:bound_CC2}) can therefore be
improved, yielding 
$$
\left| \int_\Gamma (1-\cI_{N_g})(F(|u_{N_c,N_g}|^2)-F(|u_{N_c}|^2))
\right|  \le C N_c^{3/2} N_g^{-(m+1/2-\epsilon)} \|u_{N_c,N_g}-u_{N_c}\|_{L^2_\#}.
$$
and
$$
 \|u_{N_c,N_g}-u_{N_c}\|_{H^1_\#}  \le  C N_c^{3/2} N_g^{-m}.
$$
We deduce (\ref{eq:estim_NcNg_u}) from the inverse inequality
(\ref{eq:inv_ineq}). For $N_c$ large enough, $u_{N_c,N_g}$ is bounded
away from zero, so that $f(|u_{N_c,N_g}|^2)$ is uniformly
bounded in $H^{m+1/2-\epsilon}_\#(\Gamma)$. Therefore, the $k^{\rm th}$ Fourier
coefficient of $(V^{\rm ion}+f(|u_{N_c,N_g}|^2))$ is bounded by 
$C|k|^{-m}$ where the constant $C$ does not depend on $N_c$ and $N_g$. 
Using the equality
\begin{eqnarray*}
\lambda_{N_c,N_g}-\lambda_{N_c} & = & \cN^{-1} \bigg[ \langle
(H_{|u_{N_c}|^2}-\lambda_{N_c})(u_{N_c,N_g}-u_{N_c}),(u_{N_c,N_g}-u_{N_c})
\rangle_{H^{-1}_\#,H^1_\#} \\
&& \qquad - \int_\Gamma (1-\cI_{N_g})(V^{\rm ion}+
f(|u_{N_c,N_g}|^2))|u_{N_c,N_g}|^2 \\
&& \qquad + D_\Gamma(|u_{N_c,N_g}|^2,|u_{N_c,N_g}|^2-|u_{N_c}|^2) 
+ \int_\Gamma (f(|u_{N_c,N_g}|^2)-f(|u_{N_c}|^2)) |u_{N_c,N_g}|^2 \bigg],
\end{eqnarray*}
(\ref{eq:estim_NcNg_u}) and (\ref{eq:estim_D3}),
 we obtain  (\ref{eq:estim_NcNg_lambda}). A similar calculation leads to 
(\ref{eq:estim_NcNg_I}).

Lastly, we have for all $v_{N_c} \in V_{N_c}$, 
\begin{eqnarray}
&& \!\!\!\!\!\!\!\!\!\!\!\!\!\!\!\!
E^{\rm TFW}_{N_g}(v_{N_c})- E^{\rm TFW}_{N_g}(u_{N_c,N_g}) \\
& = & \langle
(\widetilde
H_{u_{N_c,N_g}}-\lambda_{N_c,N_g})(v_{N_c}-u_{N_c,N_g}),(v_{N_c}-u_{N_c,N_g})\rangle_{H^{-1}_\#,H^1_\#}
\nonumber 
\\
&& + \frac 12 D_\Gamma(|v_{N_c}|^2-|u_{N_c,N_g}|^2,|v_{N_c}|^2-|u_{N_c,N_g}|^2)\nonumber \\
&& + \sum_{x \in \cG_{N_g} \cap \Gamma}
\left( \frac{L}{N_g} \right)^3 \left(
F(|v_{N_c}(x)|^2)-F(|u_{N_c}(x)|^2)-f(|u_{N_c}(x)|^2)(|v_{N_c}(x)|^2-|u_{N_c}(x)|^2) \right)
\nonumber \\
& \ge &   \langle
(\widetilde
H_{u_{N_c,N_g}}-\lambda_{N_c,N_g})(v_{N_c}-u_{N_c,N_g}),(v_{N_c}-u_{N_c,N_g})\rangle_{H^{-1}_\#,H^1_\#}.
\end{eqnarray}
As $u_{N_c,N_g}$ converges to $u$ in $H^2_\#(\Gamma)$, the operator $\widetilde
H_{|u_{N_c,N_g}|^2}^{N_g}-H_{\rho^0}$ converges to zero in operator
norm. Reasoning as in the proof of the uniqueness of $u_{N_c}$, we
obtain that for $N_c$ large enough and $N_g \ge 4N_c+1$, we have for all
$v_{N_c} \in V_{N_c}$ such that $\|v_{N_c}\|_{L^2_\#}=\cN^{1/2}$ and 
$(v_{N_c},u_{N_c})_{L^2_\#} \ge 0$,
$$
 \langle
(\widetilde
H_{u_{N_c,N_g}}-\lambda_{N_c,N_g})(v_{N_c}-u_{N_c,N_g}),(v_{N_c}-u_{N_c,N_g})\rangle_{H^{-1}_\#,H^1_\#}
\ge \frac \gamma 2 \|v_{N_c}-u_{N_c,N_g}\|_{H^1_\#}^2.
$$
Thus the uniqueness of $u_{N_c,N_g}$ for $N_c$ large enough.
\end{proof}

\section*{Acknowledgements} This work was done while E.C.
was visiting the Division of Applied Mathematics of Brown
University, whose support is gratefully acknowledged. 
This work was also partially supported by the ANR grant LN3M.

\end{document}